\documentclass[12pt]{amsart}
\usepackage{amsthm,amsfonts,amsmath}

\newtheorem{teorema}{Theorem}

\newtheorem{definicion}[teorema]{Definition}
\newtheorem{proposicion}[teorema]{Proposition}
\newtheorem{corolario}[teorema]{Corollary}
\theoremstyle{definition}
\newtheorem{ejemplo}[teorema]{Example}
\newtheorem{remark}[teorema]{Remark}

\title{Hypersurfaces with a canonical principal direction}
\author[E. Garnica, O. Palmas and G. Ruiz-Hern\'andez]{Eugenio Garnica\and Oscar Palmas\and Gabriel Ruiz-Hern\'andez}

\begin{document}

\maketitle

\begin{abstract}
Given a vector field $X$ in a Riemannian manifold, a hypersurface is said to have a canonical principal direction relative to $X$ if the projection of $X$ onto the tangent space of the hypersurface gives a principal direction. We give different ways for building these hypersurfaces, as well as a number of useful characterizations. In particular, we relate them with transnormal functions and eikonal equations. With the further condition of having constant mean curvature (CMC) we obtain a characterization of the canonical principal direction surfaces in 
Euclidean space as Delaunay surfaces. We also prove that CMC constant angle hypersurfaces in a product $\mathbb{R}\times N$ are either totally geodesic or cylinders.
\end{abstract}

\section{Introduction}

The detailed study of some well-known curves and surfaces lead to quite interesting recent developments in differential geometry. For example, the logarithmic spiral in $\mathbb R^2$ and the standard helix in $\mathbb R^3$ may be considered as particular cases of a more general concept, that of submanifolds making a constant angle with a given, distinguished vector field. A number of research papers in this area include \cite{MR1363411}, \cite{MR2273746}, \cite{Scala}, \cite{MR2541342}, \cite{MR2681534}, \cite{MR2785730}, \cite{MR2443735}, \cite{MR2681099}, 
\cite{MR2537560}, and references therein.

It turns out that many constant angle hypersurfaces fall into a broader class of submanifolds, defined as follows. If $\bar M^{n+1}$ is a Riemannian manifold and $M$ is an orientable hypersurface of $\bar M$, $M$ is said to have a {\em canonical principal direction relative to a vector field $X\in\mathfrak{X}(\bar M)$} if the projection of $X$ into the tangent space of the hypersurface gives a principal direction. For example, rotation hypersurfaces in Euclidean spaces have a canonical principal direction relative to a vector field parallel to its rotation axis.

Special types of canonical principal direction hypersurfaces are studied in \cite{MR2506241}, \cite{Di-Mu-Ni} and \cite{MR2772433}. We point out our recent paper \cite{gpr1}, where we work in a warped product $I\times_\varrho N$ of a real interval with a manifold $N$. In this context, we proved that a hypersurface making a constant angle with the vector field $\partial_t$ tangent to the $I$-direction has a canonical principal direction relative to $\partial_t$. Noting that the vector field $\varrho\partial_t$ is an example of a {\em closed conformal} vector field (see the condition given in equation (\ref{eq:conformal})) and as a natural continuation of \cite{gpr1}, our aim here is to study the hypersurfaces with canonical principal direction relative to a closed conformal vector field $X$.

This paper is organized as follows. In Section \ref{sec:notacion} we will give some notation and basics, while in Section \ref{sec:examples} we will give several examples.
In Section \ref{sec:main} we will prove our main result, Theorem \ref{teo:teoremon}, giving a complete characterization of these hypersurfaces. Moreover, in Section \ref{sec:construction} we prove that every hypersurface of this type is locally the graph of a transnormal function, i.e., a function satisfying a condition over the norm of its gradient (see equation (\ref{eq:eikonal})). Finally, in Section \ref{sec:particularinstances} we characterize the canonical principal direction surfaces in Euclidean space as Delaunay surfaces and prove that CMC constant angle hypersurfaces in a product $\mathbb{R}\times N$ are either totally geodesic or cylinders.

\section{Basic properties and notation}\label{sec:notacion}

Hereafter, $\bar{M}^{n+1}$ will denote a Riemannian manifold, with connection $\bar{\nabla}$. Let $M$ be an orientable hypersurface of $\bar M$ with induced connection $\nabla$ and $\xi$ a unit vector field everywhere normal to $M$.

As usual, we have the Gauss and Weingarten
equations:
\[
\bar{\nabla}_YZ=\nabla_YZ+\alpha(Y,Z),\quad
\bar{\nabla}_Y\xi=-A_\xi Y,
\]
where $Y,Z\in\mathfrak{X}(M)$, $\alpha$ is the second fundamental form of $M$ in $\bar M$, $A_\xi$ is the shape operator associated to $\xi$. Recall also that
$\alpha$ and $A_\xi$ are related by the formula
\[
\langle \alpha(Y,Z),\xi \rangle = \langle A_\xi Y, Z \rangle.
\]

Let $X\in\mathfrak{X}(\bar M)$ denote a vector field whose restriction to $M$ is transversal to $\xi$. The vector fields $X^T,T\in\mathfrak{X}(M)$ are defined by
\[
 X^T=X-\langle X,\xi\rangle\xi\quad\mathrm{and}\quad T=\frac{X^T}{\vert X^T\vert}.
\]
Note that $X^T\ne 0$ by the transversality between $X$ and $\xi$.
Finally, the angle function $\theta\in (0,\pi)$ between $X$ and $\xi$ is given by
\[
\cos\theta=\left\langle \frac{X}{\vert X\vert},\xi\right\rangle.
\]

\medskip

The vector fields which we will distinguish are given in the following definition.

\begin{definicion}
A vector field $X\in\mathfrak{X}(\bar{M})$ is {\em closed conformal} if and only if
\begin{equation}\label{eq:conformal}
\bar{\nabla}_YX=\phi Y
\end{equation}
for every $Y\in\mathfrak{X}(\bar{M})$, where $\phi$ is a
differentiable function defined on $\bar{M}$.
\end{definicion}

A constant vector field and a radial vector field in $\mathbb R^n$ are examples of closed conformal vector fields with $\phi\equiv 0$ and $\phi\equiv 1$, respectively. Closed conformal vector fields have been studied extensively in many
contexts; see \cite{MR1722814}, in particular, where S. Montiel
proved many interesting facts about them, which we collect in the
following theorem and use freely in this paper:

\begin{teorema}\label{teo:montiel}
Let $\bar{M}^{n+1}$ be a Riemannian manifold endowed with a non-null closed conformal
vector field $X$ satisfying (\ref{eq:conformal}). Then,
\begin{itemize}
\item $X$ has only zero, one, or two zeroes.

\item Away from its zeroes, $X$ defines a $n$-dimensional
distribution $X^\perp$ by taking at each point the orthogonal
complement of $X$. This distribution is integrable and each leaf
of the corresponding foliation is totally umbilical in $\bar{M}$.

\item The functions $\vert X\vert$ and $\phi$ are constant
along each leaf of the foliation.

\item Fix a connected component $N$ of a leaf of the
foliation determined by $X$ and let $\psi_t$ be the local flow of
$X$, defined in an open interval $I\subset\mathbb R$. Then the
expression
\[
\varrho(t)=\vert X_{\psi_t(p)}\vert,\quad p\in N,
\]
does not depend on the particular value chosen for $p$ and
$\bar{M}$ is locally isometric to $I\times_\varrho N$.
From this form we may recover the closed conformal vector field
$X$ as
\[
X=\vert X\vert\, \partial_t = \varrho\,\partial_t,
\]
where $\partial_t$ is the lift to $\bar{M}$ of the canonical
vector field tangent to $I$.

\end{itemize}
\end{teorema}

\section{Examples in Euclidean spaces}\label{sec:examples}

In this section we will give parametrizations of canonical principal direction hypersurfaces in Euclidean spaces.

\begin{ejemplo}
Consider the case of $\bar M=\mathbb R^3$. We will give a parametrization of a surface $M$ with canonical principal direction relative to a unit constant vector field $X_0$.

Fix a plane $N$ orthogonal to $X_0$, $\gamma=\gamma(s)$ a curve in $N$ parameterized by arc length and $\eta$ a unit vector field in $\mathfrak{X}(N)$ normal to $\gamma$. Fix a curve $\beta(t)=(f(t),g(t))$ in a $2$-dimensional plane, also parameterized by arc length. Consider the surface $M$ parameterized by
\begin{equation}\label{eq:curvascongruentes}
\varphi(s,t)=\gamma(s)+f(t)\eta(s)+g(t)X_0.
\end{equation}

The partial derivatives $\varphi_t,\varphi_s$ are given by
\[
\varphi_t=f'(t)\eta(s)+g'(t)X_0
\]
and
\[
\varphi_s=\gamma'(s)+f(t)\eta'(s)=(1-f(t)\kappa(s))\gamma'(s),
\]
where $\kappa(s)$ is the curvature of $\gamma$ at $s$. We have used the Serret-Frenet formulae and the fact that $\gamma$ is a planar curve. Note that $\varphi_t$ and $\varphi_s$ are orthogonal. Taking the unit vectors along these directions, we calculate the projection of $X_0$ onto the tangent plane of $M$ as
\[
\langle X_0,\varphi_t\rangle\varphi_t+\langle X_0,\gamma'(s)\rangle\gamma'(s)=g'(t)\varphi_t.
\]

We will suppose that $g'(t)$ does not vanish and prove that $\varphi_t$ is a principal direction of $M$.
Note that a unit vector field $\xi $ normal to $M$ is given by
\[
\xi (s,t)=-g'(t)\eta(s)+f'(t)X_0,
\]
and its partial derivative with respect to $t$ is
\[
\xi _t=-g''(t)\eta(s)+f''(t)X_0=\frac{f''(t)}{g'(t)}(f'(t)\eta(s)+g'(t)X_0)=\frac{f''(t)}{g'(t)}\varphi_t,
\]
where we have used the fact that $f'f''+g'g''=0$. The above equation implies that $\varphi_t$ is a principal direction of $M$.

The above calculation gives an expression for the principal curvature of $M$ in the direction of $\varphi_t$. As for the other principal curvature, note that
\[
\xi _s=-g'(t)\eta'(s)=g'(t)\kappa(s)\gamma'(s).
 \]

Hence, the principal curvatures $\lambda,\mu$ of a surface $M$ parameterized by (\ref{eq:curvascongruentes}) are given by
\[
\lambda=\frac{f''(t)}{g'(t)}\quad\mathrm{and}\quad \mu=\frac{g'(t)\kappa(s)}{1-f(t)\kappa(s)}.
\]

For example, consider that $\gamma(s)$ is a unit circle ($\kappa\equiv 1$) and the additional restriction of $M$ being a minimal surface; i.e., $\lambda+\mu=0$. It is easy to check that the functions $f,g$ satisfying these conditions are given by
\[
f(t)=\cosh(\sinh^{-1}(t))+1\quad \mathrm{and}\quad g(t)=\sinh^{-1}(t),
\]
which is the arc length parametrization of the catenary. Then $M$ is the standard catenoid in $\mathbb R^3$.
\end{ejemplo}

In fact, we may generalize the above example to any dimension as follows.

\begin{ejemplo}
Let $X$ be a closed conformal vector field in $\mathbb R^{n+1}$ and $N$ a hypersurface everywhere orthogonal to $X$, whose existence is guaranteed by Theorem \ref{teo:montiel}. Let $L$ be a $(n-1)$-dimensional orientable hypersurface of $N$ parameterized by $\varphi=\varphi(x_1,\dots,x_{n-1})$ and let $\eta$ be a unit vector field normal to $L$. Define a hypersurface $M^n$ of $\mathbb R^{n+1}$ by the parametrization
\[
\Phi(x,t)=\varphi(x)+f(t)\eta(x)+g(t)\widehat X(x),\quad \widehat X=X/\vert X\vert,
\]
where $x=(x_1,\dots,x_{n-1})$ and $(f(t),g(t))$ is a planar curve parameterized by arc length. The partial derivatives of the above expression are
\begin{eqnarray*}
\Phi_t & = & f'(t)\eta+g'(t)\widehat X,\\
\Phi_i & = & \varphi_i+f(t)\eta_i+g(t)\widehat X_i = \varphi_i+f(t)\eta_i+\frac{g(t)\phi}{\vert X\vert}\varphi_i,
\end{eqnarray*}
where $i=1,\dots,n-1$ and all subindices denote partial derivatives. We have used the fact that $X$ is closed conformal. Since $\langle X,\varphi_i\rangle=0$, we have
\[
\langle X,\Phi_i\rangle = f(t)\langle X,\eta_i\rangle = -f(t)\langle X_i,\eta\rangle = -f(t)\phi\langle \varphi_i,\eta\rangle = 0
\]
for each $i=1,\dots,n-1$, which means that the projection of $X$ into the tangent space of $M$ lies in the direction of $\varphi_t$. This projection vanishes iff $g'(t)=0$, so we will suppose $g'(t)\ne 0$ everywhere. Now, a unit vector field $\xi $ normal to $M$ is given by
\[
\xi =-g'(t)\eta+f'(t)\widehat X.
\]

As in the previous example, we have
\[
\xi _t=\frac{f''(t)}{g'(t)}\varphi_t,
\]
which means that $\varphi_t$ is a principal direction, implying in turn that the tangent part of $X$ determines a principal direction.
\end{ejemplo}

\newpage

\section{Characterization theorem}\label{sec:main}

In this section we state and prove our main result, giving different characterizations of the hypersurfaces with a canonical principal direction.

We will suppose that $\bar M^{n+1}$ is a Riemannian manifold is endowed with a non-null closed conformal vector field $X$ and $M$ is an orientable hypersurface of $\bar M$ with a normal unit vector field $\xi$ making a (not necessarily constant) angle $\theta\in (0,\pi)$ with $X$.

By Theorem \ref{teo:montiel}, locally $\bar M$ is isometric to $I\times_\varrho N$. In this case, we denote by $h:M\to\mathbb R$ the height function of $M$, i.e., the restriction of the projection $\pi:I\times_\varrho N\to I$ to $M$. In a region $U\subseteq N$ where the angle $\theta\ne\pi/2$, we may suppose further that $M$ is given as the graph of a function $F:U\to I$:
\[
M=\{\ (F(x),x)\ \vert \ x\in U\ \}.
\]

Note that our definition of a graph use the order of the factors according to the standard use of the notation $I\times_\varrho N$ for warped products.

\begin{teorema}\label{teo:teoremon} Let $\bar M^{n+1}$ be a Riemannian manifold admitting a non-null closed conformal vector field $X$, and $M$ an orientable hypersurface of $\bar M$. Using the notations above defined as well as those of Section \ref{sec:notacion}, the following statements are equivalent:
\begin{enumerate}
\item\label{itemcpd} $M$ has a canonical principal direction relative to $X$, i.e., $T$ is a principal direction.
\item \label{angulo} The angle $\theta$ between $X$ and $\xi$ is constant along the directions tangent to $M$ and orthogonal to $T$.
\end{enumerate}
In addition, if we consider an open subset of $\bar M$ isometric to a warped product $I\times_\varrho N$ and the angle $\theta\ne\pi/2$, the above conditions are equivalent to the following:
\begin{enumerate}
\item[3.] The integral curves of $T$ are geodesics in $M$.
\item[4.] The norm of the gradient of $h$ is constant along the level curves of $h$.
\item[5.] The norm of the gradient of $F$ is constant along the level curves of $F$.
\end{enumerate}
\end{teorema}

\begin{proof}

We decompose $X$ as
\[
X =\vert X \vert\left((\sin\theta)T+(\cos\theta)\xi\right).
\]

By differentiating the above with respect to a vector field $Z\in\mathfrak{X}(M)$, we obtain
\begin{eqnarray*}
\bar{\nabla}_Z X & = &  Z(|X|) \left[(\sin\theta)T+(\cos\theta)\xi\right] \\[0.2cm]&& +
|X|\left[Z(\sin\theta)T+(\sin\theta)\bar{\nabla}_ZT+Z(\cos\theta)\xi+(\cos\theta) \bar{\nabla}_Z\xi  \right].
\end{eqnarray*}

The left hand side of this equation is equal to $\phi Z$, since $X$ is closed conformal. Taking the tangent and normal components, we have
\begin{equation}\label{eq:tangente}
\phi Z = Z(|X|)(\sin\theta)T+
|X|\left[Z(\sin\theta)T+(\sin\theta)\nabla_ZT-(\cos\theta) A_\xi Z\right]
\end{equation}
and
\begin{equation}\label{eq:normal}
0 = Z(|X|)(\cos\theta)\xi +
|X|\left[(\sin\theta)\alpha(Z,T)+Z(\cos\theta)\xi\right].
\end{equation}

\medskip

Let $Z\in\mathfrak{X}(M)$ be orthogonal to $T$. Hence, $Z$ is also orthogonal to $X$ and by Theorem \ref{teo:montiel}, $Z(|X|)=0$. We use this fact and take the scalar product of (\ref{eq:normal}) with $\xi$ to obtain
\[
0 = (\sin\theta)\langle \alpha(Z,T), \xi\rangle +Z(\cos\theta)=(\sin\theta)\left[\langle\alpha(Z,T),\xi\rangle-Z(\theta)\right].
\]

Since $\theta\in(0,\pi)$, $\sin\theta\ne 0$ and we obtain
\[
\langle A_\xi T, Z \rangle = \langle \alpha(Z,T) , \xi \rangle=Z(\theta).
\]

Let us use these calculations to prove some of the implications announced in the statement of the theorem.

\medskip

Item \ref{itemcpd} implies item \ref{angulo}: If $T$ is a principal direction, there exists a $\lambda$ such that $A_\xi T=\lambda T$, and then the last expression implies that $Z(\theta)=\lambda\langle Z,T\rangle=0$ for every
$Z$ such that $ \langle Z,T \rangle =0$, which in turn implies that the angle $\theta$ is constant along such directions.

\medskip

Item \ref{angulo} implies item \ref{itemcpd}: If $Z(\theta)=0$ for each vector field $Z\in\mathfrak{X}(M)$ orthogonal to $T$, the last expression implies that $A_\xi T$ has no components orthogonal to $T$, and hence there exist $\lambda$ such that $A_\xi T=\lambda T$. 

\medskip

Item \ref{itemcpd} implies item 3: Suppose that there exists $\lambda$ such that $A_\xi T=\lambda T$. The expression (\ref{eq:tangente}) for $Z=T$ shows that $\nabla_TT$ is a scalar multiple of $T$, because $\sin\theta\ne 0$. But as $T$ is a unit vector field, $\langle \nabla_TT,T\rangle=0$, which means that $\nabla_TT=0$ and the integral curves of $T$ are geodesics in $M$.

\medskip

From now on we suppose that $\theta\ne\pi/2$, so that $\cos\theta\ne 0$.

\medskip

Item 3 implies item \ref{itemcpd}: Suppose that $\nabla_TT=0$. From (\ref{eq:tangente}) (for $Z=T$) and the fact that $\cos\theta\ne 0$ we conclude that one may express $A_\xi T$ as a scalar multiple of $T$. This means that $T$ is a principal direction.

\medskip


Now we will work in the warped product $I\times_\varrho N$. As usual, we will use the same notation for vector fields in every factor of this product and its corresponding liftings. Hence, if $t$ denotes a standard coordinate system on $I$, then $\partial_t$ denotes indistinctly the vector field tangent to $I$ and its corresponding lifting to $\bar M$.
Since $M$ is considered here as the graph of a function $F:N\to I$, a frame field tangent to $M$ is given by
\[
E_i=\frac{\partial F}{\partial x_i}\partial_t+e_i,\quad i=1,\dots,n.
\]
where $e_i$ denotes the lifting to $\bar M$ of a $n$-frame tangent to $N$. It is straightforward to check that the vector field defined by
\[
\xi =(\varrho\circ F)^2\partial_t-\nabla F\]
is normal to $M$. Now the height function $h:M\to\mathbb R$ is given by $h(F(x),x)=F(x)$. (Incidentally, this expression shows that each level curve of $h$ corresponds exactly to a level curve of $F$.) Since
\[
\langle \nabla h, E_i\rangle=E_i(h)=e_i(F)=\frac{\partial F}{\partial x_i} = \left\langle \partial_t,E_i \right\rangle = \left\langle \partial_t^T,E_i \right\rangle,
\]
the gradient $\nabla h$ of the height function is precisely the component of $\partial_t$ tangent to $M$. This component can be calculated as
\[
\partial_t-\frac{\left\langle \partial_t,\xi \right\rangle}{\langle \xi , \xi  \rangle }\xi =\frac{1}{\vert\nabla F\vert^2+(\varrho\circ F)^2}\left(\vert\nabla F\vert^2\partial_t+\nabla F\right).
\]

In other words, $\nabla h$ and $\nabla F$ are related by
\[
\nabla h =\frac{1}{\vert\nabla F\vert^2+(\varrho\circ F)^2}(\vert\nabla F\vert^2\partial_t+\nabla F).
\]

Hence, the relation between $\vert\nabla h\vert$ and $\vert\nabla F\vert$ is
\[
\vert\nabla h\vert^2=\frac{\vert\nabla F\vert^2}{\vert\nabla F\vert^2+(\varrho\circ F)^2}.
\]

Conversely, we may express $\vert\nabla F\vert$ in terms of $\vert\nabla h\vert$:
\[
\vert\nabla F\vert^2=\frac{(\varrho\circ F)^2\vert\nabla h\vert^2}{1-\vert\nabla h\vert^2},
\]

We will prove now the remaining claims in the theorem.

\medskip

Items 4 and 5 are equivalent: Take a level curve of $F$, which as pointed out before, corresponds precisely to a level curve of $h$. From the above expressions and the fact that $\varrho\circ F$ is constant along such a curve it is clear that $\vert\nabla F\vert^2$ is constant along the level curves of $F$ iff $\vert\nabla h\vert^2$ is constant along the level curves of $h$.

\medskip

To finish the proof, we prove the equivalence between items 3 and 4. Note, from the above considerations, that $T=\nabla h/\vert\nabla h\vert$. Now, $\nabla_TT=0$ is equivalent to
\[\nabla h\left(\frac{1}{|\nabla h|}\right) \nabla h + \frac{1}{\vert\nabla h\vert} \nabla_{\nabla h} \nabla h = 0.\]

In short, $\nabla_TT=0$ if and only if
$\nabla_{\nabla h} \nabla h$ is an scalar multiple of $\nabla h$.
For every $Y \in \mathfrak{X}(M)$ such that $ \langle Y, \nabla h \rangle = 0 $ we have
\begin{equation}
\label{derivada-de-la-norma}
Y |\nabla h|^2 = 2 \langle \nabla_Y \nabla h, \nabla h \rangle = 2\langle \nabla_{\nabla h} \nabla h, Y \rangle.
\end{equation}

Hence, $\nabla_{\nabla h} \nabla h$ is an scalar multiple of $\nabla h$ if and only if $Y |\nabla h|^2 = 0$ for every such $Y$, which happens if and only if $|\nabla h|$ is constant along the level curves of $h$.
\end{proof}

\section{Relationship with transnormal functions}\label{sec:construction}

Here we give explicit parametrizations of hypersurfaces $M$ with a canonical principal direction in a warped product $\bar
M^{n+1}=\mathbb R\times_\varrho N$. Briefly, we will prove that a
graph of a function $F:N\to\mathbb R$ has a canonical principal direction if $F$ satisfies a condition on the norm of its
gradient given in the following definition.

\begin{definicion} \label{defeikonal}
\em
Let $N$ be a Riemannian manifold and $F:N\rightarrow \mathbb R$ a differentiable function. We say that $F$
is a {\em transnormal function} if it satisfies the {\em
generalized eikonal equation}
\begin{equation}\label{eq:eikonal}
\vert\nabla F\vert=b\circ F,
\end{equation}
where $b$ is a non-negative function.
\end{definicion}

We recall that the concept of transnormal function is related to that of an isoparametric function. An isoparametric function is a transnormal function
that also satisfies the condition $\Delta F = a \circ F$, where
$a$ is a smooth function. It is well known that Cartan investigated such functions on
space forms; see \cite{MR1990032} and \cite{MR0901710} for more
details.

\medskip

In our following results we give the relation between the
transnormal functions and the hypersurfaces with a canonical principal direction.

\begin{proposicion}
\label{graph-of-transnormal}
Let $\bar{M}^{n+1}$ be the warped product $\mathbb R\times_\varrho N$. The graph of a transnormal function $F:N\to \mathbb R$ has a canonical principal direction relative to the vector field $\partial_t$.
\end{proposicion}
\begin{proof}
We denote by $\nabla F$ the lift to $\bar{M}$ of the gradient of $F$. The vector field $\xi$ everywhere normal to the
graph of $F$ is given by
\[
\xi=(\varrho\circ F)^2\partial_t-\nabla F.
\]
By Theorem \ref{teo:teoremon} we only need to analyze the angle $\theta$ between $\partial_t$ and $\xi$ along the level curves of $F$, given by
\begin{equation}\label{eq:angulo}
\cos\theta=\left\langle \frac{\xi}{\vert\xi\vert},
\partial_t\right\rangle  = \frac{\varrho\circ F}{\sqrt{(\varrho\circ F)^2+\vert\nabla F\vert^2}}.
\end{equation}

If $F$ is transnormal, substituting (\ref{eq:eikonal}) in (\ref{eq:angulo}) we obtain that $\cos\theta$ has the form $g\circ F$ and hence $\theta$ is constant along the level curves of $F$. By Theorem \ref{teo:teoremon}, we have that the graph of $F$ has a canonical principal direction.\end{proof}

In order to prove the converse of this result, we impose a natural additional condition. The statement uses the notation of the above Proposition and its proof.

\begin{proposicion}
\label{converse graph-of-transnormal}
Let $\bar{M}^{n+1}$ be the warped product $\mathbb R\times_\varrho N$. If the graph of a function $F:N\to \mathbb R$ has a canonical principal direction relative to the vector field $\partial_t$ and the angle $\theta$ between $\partial_t$ and the vector field $\xi$ normal to the graph is everywhere different from $\pi/2$, then $F$ is transnormal.
\end{proposicion}

\begin{proof}
Solving for $\vert\nabla F\vert^2$ in equation (\ref{eq:angulo}) we obtain
\[
\vert\nabla F\vert^2=(\tan^2\theta)(\varrho\circ F)^2.
\]

Suppose that the graph of $F$ has a canonical principal direction relative to $\partial_t$. Then the angle $\theta\in(0,\pi)$ is constant along the level curves of $F$. At the points in the image of $F$ we define the real-valued function $g$  as
\[
g(F(p))=\vert\tan\theta(p)\vert;
\]
this function is well-defined and differenciable, since $\theta$ is both different from $0$ and $\pi/2$. We may write then $\vert\nabla F\vert = b\circ F$, with $b=g\cdot \varrho$, i.e., $F$ is transnormal. \end{proof}

In our next result we give a solution of (\ref{eq:eikonal}) and use it to give to build a canonical principal direction hypersurface.

\begin{proposicion}
\label{proposicion:solution-eikonal} Let $N$ be a Riemannian
manifold, $L\subset N$ an orientable hypersurface of $N$ and $L_\epsilon$ a tubular neighborhood of $L$ such that the
distance function $d$ to $L$ is well-defined in $L_\epsilon$ and is
differentiable in $L_\epsilon\setminus L$. Also, let $b:I\to\mathbb R$ be a differentiable
positive function and define a real
valued, invertible function $h:I\to\mathbb R^+$ by
\begin{equation}\label{eq:h-inversa}
h^{-1}(s)=\int_{s_0}^s\frac{d\sigma}{b(\sigma)}.
\end{equation}
Then, $F=h\circ d$ satisfies $(\ref{eq:eikonal})$ in
$L_\epsilon\setminus L$.
\end{proposicion}

\begin{proof} It is well-known that the distance function $d$ satisfies $\vert\nabla d\vert=1$ in $L_\epsilon\setminus L$; then,
\begin{eqnarray*}
\vert\nabla F\vert & = & \vert\nabla(h\circ d)\vert =(h'\circ
d)\vert\nabla d\vert = h'\circ d \\ & = &
\frac{1}{(h^{-1})'(h\circ d)}=b\circ h\circ
d=b\circ F,
\end{eqnarray*}
which proves our claim. \end{proof}

Now we will analyze the uniqueness question
related with the construction of the solutions of (\ref{eq:eikonal}) given in the last Proposition.

\begin{proposicion}
\label{prop:local-solutions}
Let $b$ be a differentiable positive function and $F:N\to\mathbb R$ a solution of $(\ref{eq:eikonal})$. Then $F$ is given locally as in Proposition
$\ref{proposicion:solution-eikonal}$.
\end{proposicion}
\begin{proof}
Let $d=h^{-1} \circ F$, where $h^{-1}$ is given by equation
(\ref{eq:h-inversa}). Let us calculate the gradient of $d$ in
$\mathbb{P}$ :
\[ \nabla d = \nabla (h^{-1} \circ F)= ((h^{-1})'\circ F) \nabla F = \frac{1}{b\circ F} \nabla F.\]

Equation (\ref{eq:eikonal}) implies that $|\nabla d| =  1 $. Using Theorem 5.3 in \cite{MR2681534}, we deduce that $d$ is a distance function to a hypersurface $L\subset N$. This proves that $F=h
\circ d$ has the form given in Proposition
\ref{proposicion:solution-eikonal}.
\end{proof}

We translate these results into the language of canonical principal direction hypersurfaces.

\begin{corolario}
Let $N$ be a Riemannian
manifold, $L\subset N$ an orientable hypersurface of $N$ and $L_\epsilon$ a tubular neighborhood of $L$ such that the
distance function $d$ to $L$ is well-defined in $L_\epsilon$ and is
differentiable in $L_\epsilon\setminus L$. Also, let $b:I\to\mathbb R$ be a differentiable
positive function and define a real
valued, invertible function $h:I\to\mathbb R^+$ by
\[
h^{-1}(s)=\int_{s_0}^s\frac{d\sigma}{b(\sigma)}.
\]
Then the graph of
$F=h\circ d$ is a hypersurface in $\bar{M}=I\times_\varrho N$ with canonical principal direction. Moreover, any canonical principal direction hypersurface given as a graph of a transnormal function $F$ has this form, at least locally.
\end{corolario}

\begin{proof}
The first assertion follows from Proposition \ref{proposicion:solution-eikonal} and Proposition \ref{graph-of-transnormal}. The second follows from Proposition \ref{prop:local-solutions}.
\end{proof}

\section{Constant mean curvature hypersurfaces}\label{sec:particularinstances}

In this section we will study two particular yet important instances of canonical principal direction hypersurfaces. Firstly, we will characterize canonical principal direction hypersurfaces with constant mean curvature in Euclidean spaces as Delaunay surfaces. In the second case, we specialize to the 
situation when the ambient is a Riemannian product of the form $\mathbb R\times N$, while the hypersurface $M$ has constant mean curvature and makes a constant angle with the vector field $\partial_t$ tangent to the $\mathbb R$-direction. Under these assumptions we will prove that $M$ is totally geodesic or a cylinder.

\bigskip

Let us then consider the case of an Euclidean ambient space $\mathbb R^{n+1}$, which is obviously a warped product $\mathbb R\times\mathbb R^n$ with constant warping function $\varrho\equiv 1$.

Let $M$ be a hypersurface in $\mathbb{R}^{n+1}$ with a canonical principal direction relative to a constant vector field, say, $X_0=(1,0,\dots,0)$, given as the graph of a function $F: U \subseteq \mathbb{R}^n \longrightarrow \mathbb{R}$, where $U$ is an open set of $\mathbb{R}^n$ and $F$ is transnormal. By Proposition \ref{prop:local-solutions} there is an orientable hypersurface $L$ in $\mathbb R^n$ such that $F=h\circ d$, where $d$ is the distance function to $L$ and $h$ satisfies equation (\ref{eq:h-inversa}). In fact, any such $L$ is just a level surface of $F$.

Now take a unit vector field $\eta$ normal to the hypersurface $L$ in $\mathbb R^n$ and consider the points $q$ of the form $q=p+t\eta(p)$, $p\in L$. In an adequate tubular neighborhood of $L$ in $\mathbb R^n$, we have
\[
F(q)=h\circ d(q) = h(t),
\]
which means that the intersection of $M$ (i.e., the graph of $F$) with the plane passing through $(0,p)\in\mathbb R\times\mathbb R^n$ and spanned by $(0,\eta(p))$ and $X_0$ is precisely the graph of $h$. In particular, the intersections of $M$ with these planes are all congruent.

\begin{remark}\label{rem:congruencia} The above argument and a reparametrization of the graph of $h$ by arc length show that every surface in $\mathbb R^3$ with canonical principal direction relative to a constant vector $X_0$ given as the graph of a transnormal function has a parametrization of the form given by equation (\ref{eq:curvascongruentes}) in Section \ref{sec:examples}, namely,
\[
\varphi(s,t)=\gamma(s)+f(t)\eta(s)+g(t) X_0,
\]
where $\gamma$ is a planar curve and $\eta(s)$ is a unit vector field normal to $\gamma$.
\end{remark}

In short, $M$ is locally constructed by considering an orientable hypersurface $L$ in $\mathbb R^n$ with an unit normal vector field $\eta$ and by putting a copy of a planar curve $\gamma$ in the plane passing through $(0,p)\in\mathbb R\times\mathbb R^n$, $p\in L$,
determined by $\eta$ and $X_0$. Let us observe that $T$, the unit vector field in the direction of the projection of $X_0$ into the tangent space of $M$, is a vector field tangent to each of these copies of $\gamma$. 

\begin{proposicion}
\label{prop:CMC}
Let $M$ be a hypersurface in $\mathbb{R}^{n+1}$ with a canonical principal direction relative to $X_0=(1,0,\dots,0)$, given as a graph of a transnormal function $F:U\subset\mathbb R^n\to\mathbb R$. If $t\in\mathbb R$ is such that $F^{-1}(t)\ne\emptyset$, let
\[
M_t:= \{ t \} \times F^{-1}(t)
\]
be a slice of $M$. If $M$ has constant mean curvature in $\mathbb{R}^{n+1}$ then
\begin{itemize}
\item $M_t$ has constant mean curvature in $\mathbb{R}^{n+1}$ for every $t$, and
\item $M_t$ has constant mean curvature in $M$ for every $t$.
\end{itemize}
\end{proposicion}

\begin{proof}
Let $p=(t,x) \in M_t$, i.e., $F(x)=t$. A basic property of the second fundamental forms for $M_t \subset M \subset \mathbb{R}^{n+1}$ says that
\begin{equation}
\label{eqn:secondff}
 \overline\alpha_t(X,Y) = \alpha_t(X,Y) + \alpha(X,Y)   ,
\end{equation}
for every $X,Y \in T_pM_t$. Here $\overline\alpha_t, \alpha_t, \alpha$ are the second fundamental forms of $M_t$ in $\mathbb{R}^{n+1}$, that of $M_t$ in $M$ and that of $M$ in $\mathbb{R}^{n+1}$, respectively.

Let $X_1, \ldots , X_n$
be a orthonormal basis of $T_pM_t$. 
Therefore $$X_1, \ldots , X_n, T$$ is a orthonormal basis of $T_pM$. By definition of the corresponding mean curvature vectors we have that
\[
\overline H_t= \sum_{i=1}^n \overline\alpha_t(X_i, X_i), \ H_t= \sum_{i=1}^n \alpha_t(X_i, X_i), \ H= \sum_{i=1}^n \alpha(X_i, X_i) + \alpha(T,T).
\]

By equation (\ref{eqn:secondff}), the mean curvature vectors are related by
\begin{equation}
\label{eqn:mcvf}
\overline H_t=H_t + H - \alpha(T,T).
\end{equation}

If $\xi$ is a unit vector field normal to $M$, we observe that $H$ and $\alpha(T,T)$ are scalar multiples of $\xi$, while $\overline H_t$ is a scalar multiple of the lift of the gradient $\nabla F$
and $H_t$ is a multiple of $T$. By Theorem \ref{teo:teoremon}, the angle $\theta \ne 0$ between $\xi$ and $X_0$ is constant along $M_t$. This implies that the angle $\theta\pm\pi/2$ between $\nabla F$ and $\xi$ is also constant along $M_t$ and different from $\pm\pi/2$. So, we deduce from (\ref{eqn:mcvf}) that
\begin{eqnarray*}
\cos(\theta\pm\pi/2) \vert \overline H_t\vert & = & \langle \overline H_t , \xi \rangle = \langle H_t, \xi \rangle +
\langle H, \xi \rangle - \langle \alpha(T,T), \xi \rangle\\ & = & \vert H\vert -\vert \alpha(T,T)\vert,
\end{eqnarray*}
where $\cos(\theta\pm\pi/2)\ne 0$. By Theorem \ref{teo:teoremon}, the integral curves of $T$ are geodesics in $M$, which implies
\[
\vert \alpha(T,T)\vert  = \vert \nabla_T T + \alpha(T, T) \vert  =  \vert \bar\nabla_T T \vert,
\]
where as before $\nabla$ is the Levi-Civita connection of $M$ and $\bar\nabla$ is the standard Levi-Civita connection of $\mathbb{R}^{n+1}$. We observe that the last term $\vert \bar\nabla_T T \vert$ is constant along $M_t$, because it measures the curvature in $\mathbb{R}^{n+1}$ of the integral
curves of $T$, which we have seen to be congruent. Since $\vert H\vert$ is constant by hypothesis, the mean curvature $\vert \overline H_t\vert$ of $M_t$ in $\mathbb{R}^{n+1}$ is constant for each $t$.

On the other hand, we take the scalar product of (\ref{eqn:mcvf}) with $X_0$ to obtain
\begin{eqnarray*}
0 & = & \langle \overline H_t, X_0  \rangle = \langle H_t , X_0  \rangle + \langle H, X_0  \rangle
-\langle \alpha(T,T), X_0  \rangle \\[0.1cm]
  & = & \cos(\theta\pm\pi/2) \vert H_t\vert  + (\cos\theta)(\vert H\vert  - \vert \alpha(T,T)\vert).
\end{eqnarray*}

Here, we are using that $\nabla F$ is orthogonal to $X_0$. Since the second summand is constant along $M_t$ and $\cos(\theta\pm\pi/2)\ne 0$, the mean curvature $\vert H_t\vert$ of $M_t$ in $M$ is constant, which concludes our proof. \end{proof}

Munteanu and Nistor proved in \cite{MR2772433} that the only minimal surface in $\mathbb R^3$ with a canonical principal direction (relative to a constant vector field) besides the plane is the catenoid. We extend their result to the constant mean curvature case and the Delaunay surfaces, that is, those surfaces of revolution with constant mean curvature.

\begin{corolario}
Let $M$ be an immersed connected surface in $\mathbb{R}^3$ with canonical principal direction relative to a constant vector field.
If $M$ has constant mean curvature then $M$ is (part of) a Delaunay surface.
\end{corolario}

\begin{proof}
We can assume that $M$ has canonical principal direction with respect to the constant vector field $(1,0,0)$.
Proposition \ref{converse graph-of-transnormal} proves that $M$ is locally the graph of a transnormal function.
By Proposition \ref{prop:CMC}, every slice $M_t$ has constant mean curvature in $\mathbb R^3$, which means that each connected component of $M_t$ is a planar curve $\gamma$ with constant curvature, that is, a line segment or a circular arc. Recall also (see Remark \ref{rem:congruencia}) that $M$ is reconstructed by attaching to each $\gamma(s)$ a curve $\beta(t)=(f(t),g(t))$ contained in the plane orthogonal to $\gamma(s)$. That is, we attach to $\gamma$ a family of congruent curves.

Suppose that $\gamma$ is a line segment. By attaching to it copies of $\beta$ we obtain that locally $M$ is a cylinder over $\beta$. Since $M$ has constant mean curvature, the curvature of $\beta$ is constant as well and again, it must be a line segment or a circular arc. In the first case, $M$ is locally a plane, while in the second is locally a right circular cylinder.

On the other hand, if $\gamma$ is a circular arc, attaching copies of $\beta$ to it produces a surface of revolution. Since $M$ has constant mean curvature, it is part of a Delaunay surface.
\end{proof}

In particular, if $M$ is a connected, complete, minimal surface in $\mathbb R^3$ with canonical principal direction relative to a constant vector field, then it must be a catenoid or a plane.

\bigskip

We conclude this paper by characterizing constant angle hypersurfaces in a Riemannian product.  We will use the Bochner formula, valid for any smooth function $F$ over a Riemannian manifold:
\begin{equation}
\label{eqn:bochner}
\frac{1}{2} \Delta |\nabla F|^2 = \langle \nabla F, \nabla \Delta  F  \rangle - \operatorname{Ric}(\nabla F, \nabla F) - |\operatorname{Hess} F|^2.
\end{equation}
See \cite{MR2088027} for the proof of this important formula, as well as Sakai's paper \cite{MR1374461} whose ideas inspired the proof of our result.

\begin{teorema}
Let $N$ be a Riemannian manifold with nonnegative Ricci curvature.
Let $M \subset \mathbb{R} \times N $ be an immersed hypersurface making a constant angle with the vector field $\partial_t$ tangent to the $\mathbb R$-direction. 
If $M$ has constant mean curvature then $M$ is either totally geodesic or it is part of a cylinder over
a constant mean curvature hypersurface immersed in $N$.
\end{teorema}

\begin{proof}
By hypothesis the angle $\theta$ between $M$ and $\partial_t$ is constant. If $\theta=0$, then $\partial_t$
is tangent to $M$, i.e., $M$ is foliated by the integral lines of $\partial_t$, which are known to be geodesics in $\mathbb{R} \times N$. This says that $M$ is part of the 
cylinder $\mathbb{R}\times\pi(M)$, where $\pi: \mathbb{R}\times N \to \mathbb{R}$ is the natural projection and $\pi(M)$ is 
an immersed hypersurface in $N$. Since $M$ has constant mean curvature, the same happens with $\pi(M)$.

If $\theta\ne 0$, then 
$M$ is locally the graph of a function 
$F: U \subset N \to \mathbb{R}$. The main result from our previous work \cite{gpr1} states that $|\nabla F| =c$, a constant; i.e., $F$ 
is an eikonal function. On the other hand, it is well known that the mean curvature vector $H$ of the graph of a function is given by
$$H=\operatorname{div}\left(\frac{\nabla F}{(1+ |\nabla F|^2)^{1/2}}\right)$$
Since $M$ has constant mean curvature and $F$ is eikonal, we conclude that
$\Delta F=\operatorname{div}\nabla F $ is a constant function. So, we have that $F$ is eikonal with constant Laplacian.
Using Bochner formula (\ref{eqn:bochner}) and the hypothesis on the Ricci curvature of $N$ we conclude that $\operatorname{Hess} F $ vanishes identically. Since $\operatorname{Hess} F(X,Y)= \langle \nabla_ X \nabla F , Y \rangle$, we have that  $\nabla F$ is a parallel vector field of $N$.

This fact implies in turn that the level hypersurfaces $L_t=F^{-1}(t)$ of $F$ are totally geodesic in $N$; then  
$\{ t \}\times L_t$ is totally geodesic in $\mathbb{R}\times N$ because every $\{t\}\times N$ is 
totally geodesic in $\mathbb{R}\times N$. We conclude that  $M$ is foliated by the totally geodesic
hypersurfaces $\{ t \}\times L_t$. 

On the other hand, consider the vector field over $M$ given by
$$T= \frac{|\nabla F|^2 \partial_t + \nabla F}{|\nabla F|(1+  |\nabla F|^2)^{1/2}},$$
which is just the unit vector field in the direction of the component of $\partial_t$ tangent to $M$. Here we are using the lift of $\nabla F$ into the Riemannian product $\mathbb{R}\times N$. Since $\vert\nabla F\vert$ is constant, we may write the above as
$$T=c_1\partial_t+c_2\nabla F,$$
where $c_1,c_2$ are constant. Therefore, $T$ is parallel in $M$ because it is a sum of parallel vector fields in $\mathbb{R}\times N$.
So, the integral lines of $T$ are geodesics in the above product.
This proves that $M$ is totally geodesic.
\end{proof}

\bibliographystyle{plain}
\bibliography{references}

\end{document}